\documentclass[a4paper,12pt]{article}

\usepackage[english,french]{babel}
\RequirePackage[utf8]{inputenc} %Caractères spéciaux
\RequirePackage[section]{placeins}%Pour placement de section
\RequirePackage[T1]{fontenc} %Quelques lettres qui sont pas inclus dans UTF-8
\RequirePackage{mathtools} %Paquet pour des équations et symboles mathématiques
\RequirePackage{siunitx} %Pour écrire avec la notation scientifique (Ex.: \num{2e+9})
\RequirePackage{float} %Pour placement d'images
\RequirePackage{graphicx} %Paquet pour insérer des images
\RequirePackage[justification=centering]{caption} %Pour les légendes centralisées
\RequirePackage{subcaption}
\RequirePackage{wallpaper}
\RequirePackage{nomencl}
\RequirePackage{fancyhdr}
\RequirePackage{url}
\RequirePackage[hidelinks]{hyperref}%Paquet pour insérer légendes dans des sous-figures comme Figure 1a, 1b
\RequirePackage[left=2.5cm,right=2.5cm,top=2cm,bottom=3.5cm]{geometry} %Configuration de la page
\usepackage{float}
\usepackage[T1]{fontenc}
\usepackage[utf8]{inputenc}
\usepackage{authblk}
\usepackage{lipsum}
\usepackage{multirow}
\usepackage{textcomp}
\usepackage{colortbl}
\usepackage{accents}
\usepackage{amsmath}
\usepackage{amssymb}
\usepackage{amsthm}
\usepackage{gensymb}
\usepackage{pifont}
\usepackage{pdfpages}
\usepackage{mathrsfs}
\usepackage{stmaryrd}
\usepackage{subcaption}
\usepackage{version}

\usepackage{tcolorbox}
\usepackage{stmaryrd}
\usepackage{multicol}
\usepackage{pdfpages}
\usepackage{mathtools}
\usepackage{algorithm}
\usepackage{algorithmic}
\usepackage{tikz}
\usepackage{thmtools}%
\usepackage{bm}
\usepackage{setspace}

\DeclareMathAlphabet{\mathbbold}{U}{bbold}{m}{n}
\theoremstyle{plain}

\declaretheorem[name=Proposition, numberwithin=section]{prop}

\declaretheorem[name=Definition, numberwithin=section]{definition}

\newcommand{\Real}{\mathbb{R}}
\newcommand{\Proba}{\mathbb{P}}
\newcommand{\Xdom}{\mathcal{X}}
\newcommand{\Udom}{\mathcal{U}}

\newcommand{\Cdom}{\mathcal{C}}
\newcommand{\Ddom}{\mathcal{D}}

\newcommand{\Fdom}{\mathcal{F}}
\newcommand{\Esp}{\mathbb{E}}

\newcommand{\UE}[1]{\renewcommand{\UE}{#1}}
\newcommand{\sujet}[1]{\renewcommand{\sujet}{#1}}
\newcommand{\titre}[1]{\renewcommand{\titre}{#1}}
\newcommand{\enseignant}[1]{\renewcommand{\enseignant}{#1}}
\newcommand{\eleves}[1]{\renewcommand{\eleves}{#1}}

\usepackage[normalem]{ulem}

\makeatletter
\newskip\@bigflushglue \@bigflushglue = -100pt plus 1fil

\def\bigcentering{\let\\\@centercr\rightskip\@bigflushglue%
\leftskip\@bigflushglue
\parindent\z@\parfillskip\z@skip}

\makeatother

\title{Sensitivity Analysis for pollutant concentration maps} %Titre du fichier
\usepackage[style=authoryear-comp,maxbibnames=9,maxcitenames=2,uniquelist=false,backend=biber,url=false]{biblatex}
\newrobustcmd*{\parentexttrack}[1]{%
  \begingroup
  \blx@blxinit
  \blx@setsfcodes
  \blx@bibopenparen#1\blx@bibcloseparen
  \endgroup}

\AtEveryCite{%
  }

\makeatother

\title{Sensitivity analysis for sets : application to pollutant concentration maps}
\author[1,2]{N. Fellmann}
\author[2]{M. Pasquier}
\author[1]{C. Blanchet-Scalliet}
\author[1]{C. Helbert}
\author[2]{A. Spagnol}
\author[2]{D. Sinoquet}
\affil[1]{École Centrale de Lyon, CNRS UMR 5208, Institut Camille Jordan, 36 Avenue Guy de Collongue,
69134 Écully, France, }
\affil[2]{IFP Energies Nouvelles}
\affil[ ]{\textit{\{noe.fellmann,christophette.blanchet,celine.helbert\}@ec-lyon.fr}}
\affil[ ]{\textit{\{noe.fellmann,mathis.pasquier,delphine.sinoquet,adrien.spagnol\}@ifpen.fr}}
\begin{document}

\selectlanguage{english}
\date{}
\maketitle

\begin{abstract}
In the context of air quality control, our objective is to quantify the impact of uncertain inputs such as meteorological conditions and traffic parameters on pollutant dispersion maps. It is worth noting that the majority of sensitivity analysis methods are designed to deal with scalar or vector outputs and are ill suited to a map-valued output space. To address this, we propose two classes of methods. The first technique focuses on pointwise indices. Sobol indices are calculated for each position on the map to obtain Sobol index maps. Additionally, aggregated Sobol indices are calculated. Another approach treats the maps as sets and proposes a sensitivity analysis of a set-valued output with three different types of sensitivity indices. The first ones are inspired by Sobol indices but are adapted to sets based on the theory of random sets. The second ones adapt \textit{universal} indices defined for a general metric output space. The last set indices use kernel-based sensitivity indices adapted to sets. The proposed methodologies are implemented to carry out an uncertainty analysis for time-averaged concentration maps of pollutants in an urban environment in the Greater Paris area. This entails taking into account uncertain meteorological aspects, such as incoming wind speed and direction, and uncertain traffic factors, such as injected traffic volume, percentage of diesel vehicles, and speed limits on the road network.
\end{abstract}

%%----Intro----%%
\newpage
\section{Introduction }

Air quality and dispersion of atmospheric pollutants are key aspects tackled by the emerging discipline of urban physics. Using computational fluid dynamics (CFD) as a tool, this branch of physics simulates pollutant dispersion to aid in urban planning and emergency decision-making. In \cite{Pasquier.2023}, a complete modelling chain has been developed to simulate traffic-related pollutant dispersion at the local urban scale. The method combines a microscopic traffic simulator with a physical engine model to estimate realistic road emissions, which are used as input of a CFD code to model unsteady atmospheric dispersion and compute two-dimensional time-averaged concentration maps at ground level. A presentation of this application is displayed in figure \ref{CFDoutputs}.

\begin{figure}[h!]
    \begin{subfigure}{0.475\textwidth}
        \centering
        \includegraphics[width=1\textwidth]{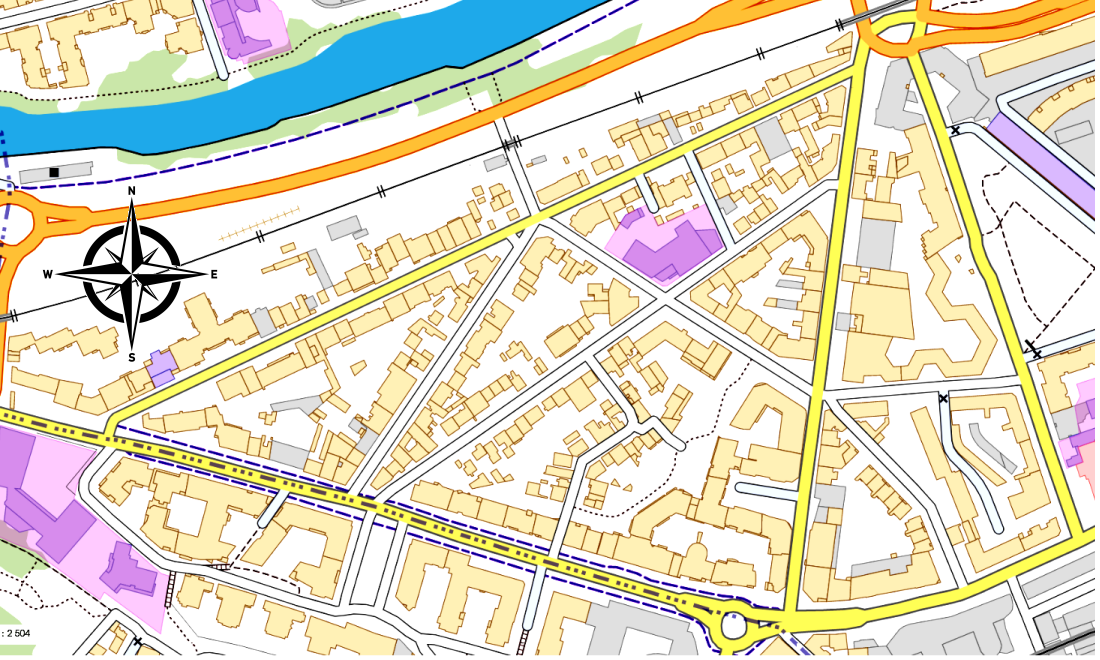}
        \caption{}
        \label{}
    \end{subfigure}
    \begin{subfigure}{0.475\textwidth}
        \centering
        \includegraphics[width=1\textwidth]{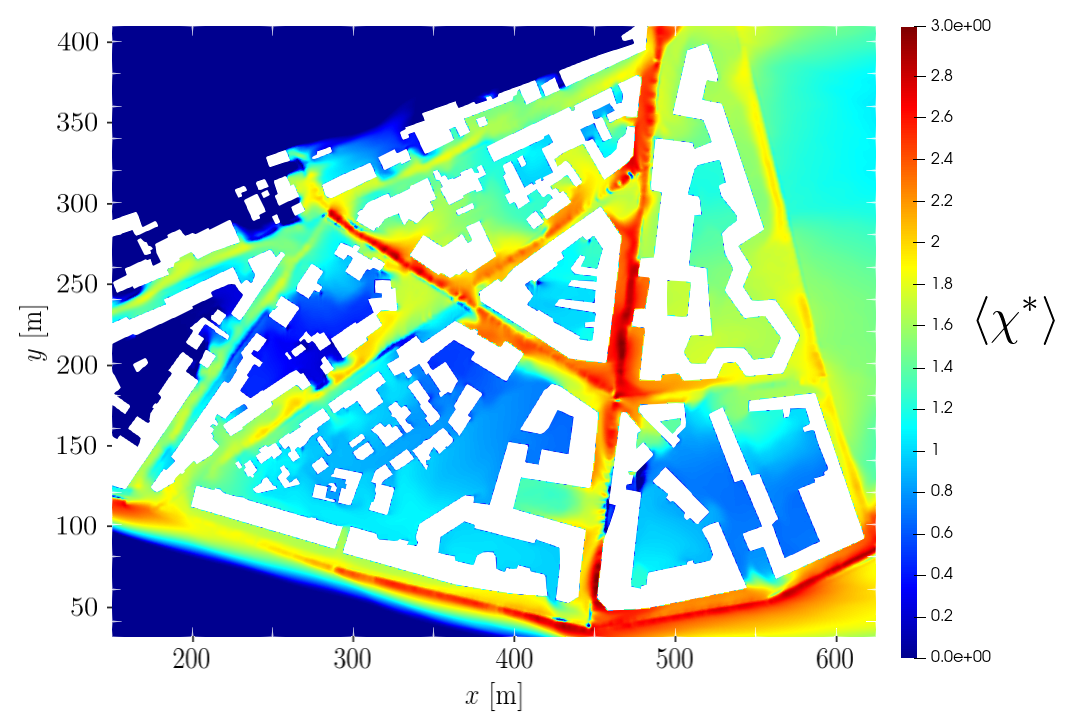}
        \caption{}
        \label{}
    \end{subfigure}
    \caption{Concentration maps from CFD numerical simulations of pollutant dispersion at the urban scale (\cite{Pasquier.2023}). \\ (a) Domain of interest, captured from OpenStreetMaps. \\(b) Time-averaged map of the logarithm of the concentration, denoted $\langle\chi^*\rangle$, obtained from a run of the CFD code with realistic road emissions. The wind is coming from the left in this case.}
    \label{CFDoutputs}
\end{figure}

The modelling chain outputs two-dimensional concentration maps from a set of inputs associated with meteorological and traffic-related uncertainties. These five uncertain inputs are the incoming wind direction $\theta$ [radians] and intensity $U_\infty$ [metres/second], together with the volume of traffic injected into the road network in the urban geometry, $q$ [vehicles/hour], the proportion of diesel and petrol engines in the fleet $\beta$ [-], and the speed limit over the network, $\nu_{max}$ [kilometres/hour]. The modelling chain with the corresponding uncertain variables is displayed in figure \ref{UQworkflow} and the probability distributions associated with each variable are given in table \ref{tab:UQdistributions} and plotted in figure \ref{UQdistributions}.

\begin{figure}[h!]
    \centering
    \includegraphics[width=0.7\textwidth]{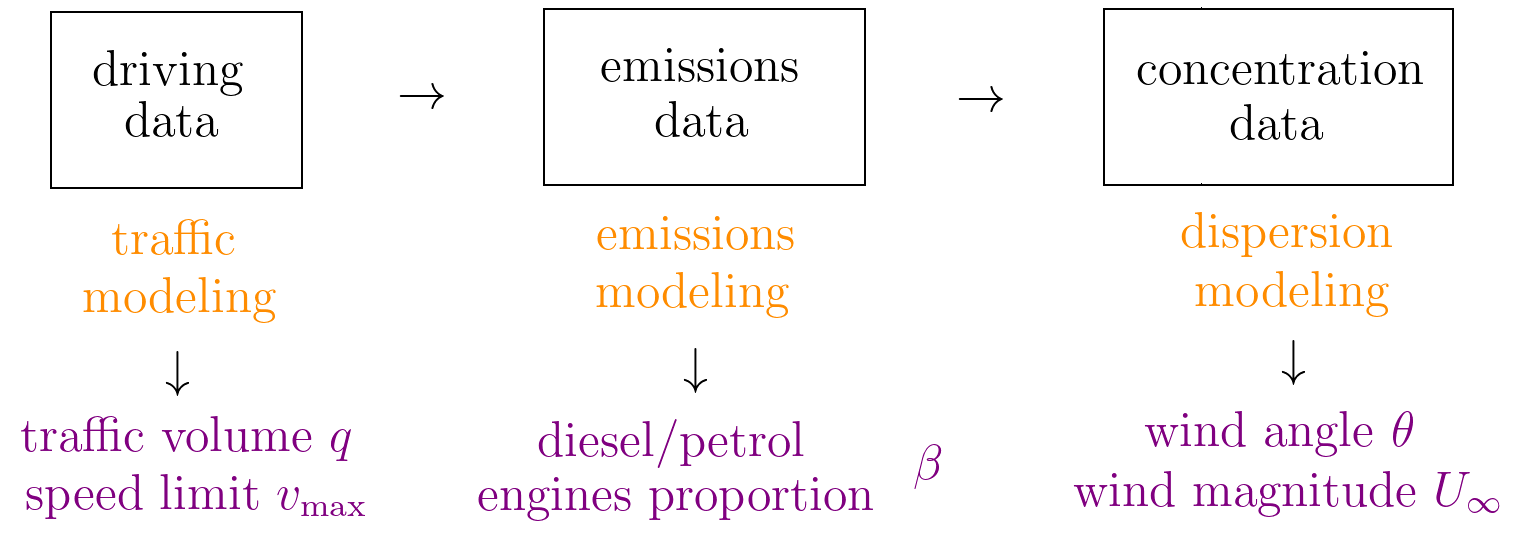}
    \caption{Whole modelling chain designed in \cite{Pasquier.2023} for the calculation of time-averaged concentration maps. The five uncertain parameters considered are indicated under each block.}
    \label{UQworkflow}
\end{figure}

\begin{table}[h!]
    \centering
    \begin{tabular}{|c|c|c|c|}\hline
        variable & distribution & bounds & parameters \\\hline
        $\theta$ & truncated normal & $[5;10]$ [rad] & $\mu=0.7,\sigma=0.5$ \\\hline
        $U_\infty$ & truncated normal & $[-0.35;1.75]$ [m/s] & $\mu=8,\sigma=2$ \\\hline
        $q$ & truncated skewed normal & $[100;500]$ [vehicle/h] & $\xi=450,\omega=100,\alpha=-3$ \\\hline
        $\beta$ & uniform & $[0;1]$ [-] & - \\\hline
        $\nu_{max}$ & uniform & $[30;50]$ [km/h] & - \\\hline
    \end{tabular}
    \caption{Probability distributions and associated parameters for the five uncertain pollutant dispersion variables}
    \label{tab:UQdistributions}
\end{table}

\begin{figure}[h!]
    \centering
    \includegraphics[width=1\textwidth]{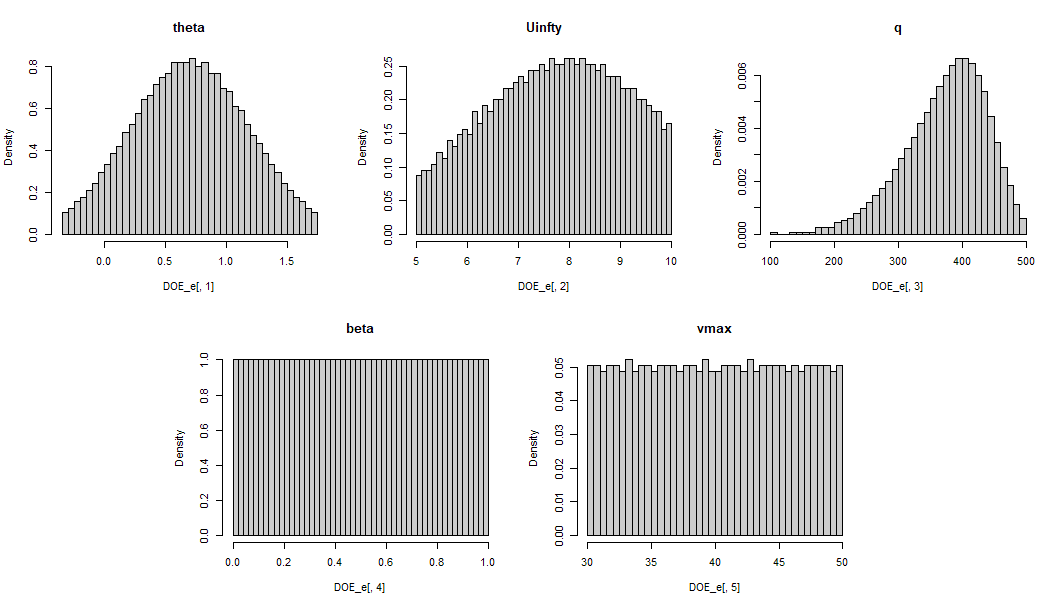}
    \caption{Latin hypercube samples of size $10^4$ of the probability distributions chosen for the five uncertain variables considered in the pollutant dispersion case.}
    \label{UQdistributions}
\end{figure}

To efficiently perform a global sensitivity analysis (SA) to compare the relative influence of these uncertain inputs on the output maps, we construct a metamodel of the entire modelling chain using a combination of principal component analysis and Gaussian process regression (GPR), known in the literature as POD-GPR (\cite{Guo.2018,Nony.2023,Marrel.2014}). The metamodel is trained on a Latin Hypercube Sampling (LHS) design of experiments of $150$ expensive CFD simulations -- each simulation requires about $7$ hours of computational time on 960 cores on the high performance computing facility of IFP Energies Nouvelles -- and validated on a set of $85$ additional simulations. A global coefficient of determination, defined as in \cite{Nony.2023}, is computed to quantitatively evaluate the predictivity of the POD-GPR metamodel, which reaches about $86\%$. The lack of predictivity is related to residual noise in the numerical simulation results due to the difficult statistical convergence of the time-averaged concentration fields in CFD simulations.

Based on this cheap metamodel, we want to perform a global sensitivity analysis for concentration maps, which can be seen as instances of a functional output from $\Ddom \subset \Real^2$ to $\Real$. Initially, sensitivity analysis focused primarily on models with scalar inputs and outputs. Standard methods such as Morris analysis (\cite{morris}) and the Sobol' index (\cite{sobol_global_2001}) were widely used to evaluate the variation of outputs in response to variations in input parameters. Morris analysis, for example, relies on sampling scenarios in which each parameter is modified one at a time to assess its effect on the model's output. However, as modelling and complex systems have evolved, the applications of sensitivity analysis have expanded to include more complex inputs and outputs. Inputs can be vectors, categorical data, time series and more. Outputs can be multi-dimensional, non-scalar or even probability distributions. To meet these needs, new advanced sensitivity analysis techniques have emerged. For example, Sobol' indices have been generalised for vectorial and functional outputs (\cite{vectorial_sobol}). However, the latter uses a truncated expansion of the functional output to then reuse the proposed index on vector outputs, which implies a loss of information. In \cite{GSAspatial}, the authors also propose a method to conduct a global sensitivity analysis in the case of spatial output through functional principal component analysis. \cite{SAINTGEOURS2014153} also perform sensitivity analysis for spatial output by computing maps of sensitivity indices. In \cite{GSA_Wass_universal_index}, the authors propose universal sensitivity indices that can be computed in any metric output space.  The more recent indices based on the Hilbert-Schmidt Independence Criterion (HSIC) were also initially defined for scalar and vector inputs and outputs (\cite{Gretton_hsic1} and \cite{daveiga2013global}), but are much more permissive and can be extended to arbitrary outputs as long as a kernel is defined on the output space. For example, HSIC-based indices have been used for probability distribution outputs in \cite{daveiga:hal-03108628}.
%(see \cite{fellmann2023kernelbased} for set-valued sensitivity analysis, for example).  

The aim of this paper is to propose and compare different methods for carrying out a sensitivity analysis of a model that produces two-dimensional maps.   We first adopt a pointwise approach, specifically a sensitivity analysis technique using Sobol' indices computed at each point of the map. The resulting Sobol' index maps can be interpreted directly or used to compute \textit{generalised} Sobol' indices, similar to those proposed in \cite{vectorial_sobol}, where the discretised map is considered as a vector. The other proposed approach analyses the spatial output without discretisation and instead aims to establish an index that quantifies the influence of each input on the overall spatial output.  To achieve this, three new SA indices based on sets are introduced and adapted to maps. Each index is tested on the POD-GPR metamodel.

The paper is structured as follows. Section \ref{sec:pointwiseSA} proposes a pointwise method for performing sensitivity analysis on maps, where section \ref{subsec:notations} establishes essential notations and section \ref{subsec:pointwiseSA} defines Sobol' indices at each point of the maps and proposes \textit{generalised} indices that aggregate all pointwise indices. Section \ref{sec:set} is devoted to three new set-based SA indices adapted to maps. Section \ref{subsec:notations2} provides the necessary notations for dealing with sets, while section \ref{subsec:vorob} deals with sensitivity analysis based on random set theory. Section \ref{subsec:univ} adapts universal indices from \cite{GSA_Wass_universal_index} to set-valued outputs, and in section \ref{subsec:kernel} we use indices based on kernels and the Hilbert-Schmidt independence criterion in the context of spatial outputs, as done in \cite{fellmann2023kernelbased}. Section \ref{sec:comparison} is devoted to a comparative study of the various sensitivity indices introduced. Finally, our conclusion (section \ref{sec:conclusion}) ties these sections together and summarises our main findings. %and discussing potential directions for future research.

%%----Présentation du modèle----%%
\newpage
\section{Pointwise sensitivity analysis for maps}
\label{sec:pointwiseSA}
\subsection{Notations}
\label{subsec:notations}
Let us define a map-valued model $\Phi$ by
$$
\begin{matrix}
 & \Udom & \longrightarrow & \Fdom (\Ddom,\Real) \\ 
 \Phi : & \bm u & \mapsto & \Phi_{\bm u} : x_1, x_2 \mapsto \Phi_{\bm u}(x_1, x_2)
\end{matrix}
$$
with $\Ddom \subset \Real^2$, $\Udom = \Udom_1 \times ... \times \Udom_p \subset \Real^p$ and $\Fdom (\Ddom,\Real)$ the space of functions from $\Ddom$ to $\Real$. As in many sensitivity analysis framework, the inputs are assumed to be a random vector $\bm U =(U_1,...,U_p) \in \Udom$ with independent components. It is defined on a probability space $(\Omega , \Fdom, \Proba)$ with known distributions $\Proba_{\bm U} = \bigotimes_{i=1}^p\Proba_{U_i}$. 

In the case of the POD-GPR metamodel of pollutant concentration maps, we have 
\begin{itemize}
    \item $\Udom= [5;10]\times[-0.35;1.75]\times[100;500]\times[0;1]\times[30;50]$ and the $\Proba_{U_i}$ are given in Table \ref{tab:UQdistributions}.
    \item $\Ddom= [0;550]\times [0;400]$ 
    \item $\Phi_{\bm u}(x_1, x_2)$ is the pollutant concentration at position $x_1, x_2 \in \Ddom$ in the physical domain when the uncertain parameters are equal to $\bm u \in \Udom$.
\end{itemize}

\subsection{Sobol' indices maps and generalized Sobol' indices}
\label{subsec:pointwiseSA}
First, we propose to analyse the sensitivity of the concentration fields to the uncertain inputs by calculating the Sobol' indices of the spatial maps. In other words, for each $x_1, x_2 \in \Ddom$, we compute the Sobol' indices of the output concentration $\Phi_{\bm U}(x_1, x_2)$.

At each point $x_1, x_2$, from the ANOVA decomposition of $\bm u \mapsto \Phi_{\bm u}(x_1, x_2)$, a variance decomposition of $\Phi_{\bm U}(x_1, x_2)$ can be obtained. The first-order Sobol' indices are defined from the first terms of the ANOVA decomposition, each of which depends on a single uncertain variable. The first-order index corresponding to the $i$-th uncertain variable is defined as 
\[
    S_i(x_1, x_2) = \frac{\operatorname{Var}[\mathbb{E}[\Phi_{\bm U}(x_1, x_2)\lvert U_i]]}{\operatorname{Var}[\Phi_{\bm U}(x_1, x_2)]}.
\]
These indices indicate the proportion of the total variance that can be attributed to the individual contribution of each uncertain variable $U_i$. 

Such first-order indices are computed using the standard partial variance estimator recalled in \cite{Saltelli.2010} and defined as
\begin{equation}
\label{eq:pickandfreeze}
    \operatorname{Var}[\mathbb{E}[\Phi_{\bm U}(x_1, x_2)\lvert U_i]] \approx \frac{1}{N} \sum_{j=1}^N \Phi_{\bm U'^{(j)}}(x_1, x_2) \left(\Phi_{\widetilde{\bm U}^{i(j)}}(x_1, x_2) - \Phi_{\bm U^{(j)}}(x_1, x_2)\right), 
\end{equation}
where $(\bm U^{(j)},\bm U'^{(j)})_{j=1}^N$ are independent and identically distributed (iid) samples of $(\bm U, \bm U')$, with $\bm U'$ being an independent copy of $\bm U$. The matrix $\widetilde{\bm U}^{i }$ is the matrix $\bm U$ with the $i$-th column replaced by that of $\bm U'$. The denominator $\operatorname{Var}[\Phi_{\bm U}(x_1, x_2)]$ is estimated using the classical unbiased sample variance estimator. The indices are estimated for each location $\left\{x^{(j)}_1, x^{(j)}_2\right\}_{j=1}^m$ of a discretised computational grid composed of $m$ nodes to obtain two-dimensional first-order sensitivity maps, as shown in Figure \ref{firstOrderSobolMaps}.

\begin{figure}[H]
    \centering
    \includegraphics[width=1\textwidth]{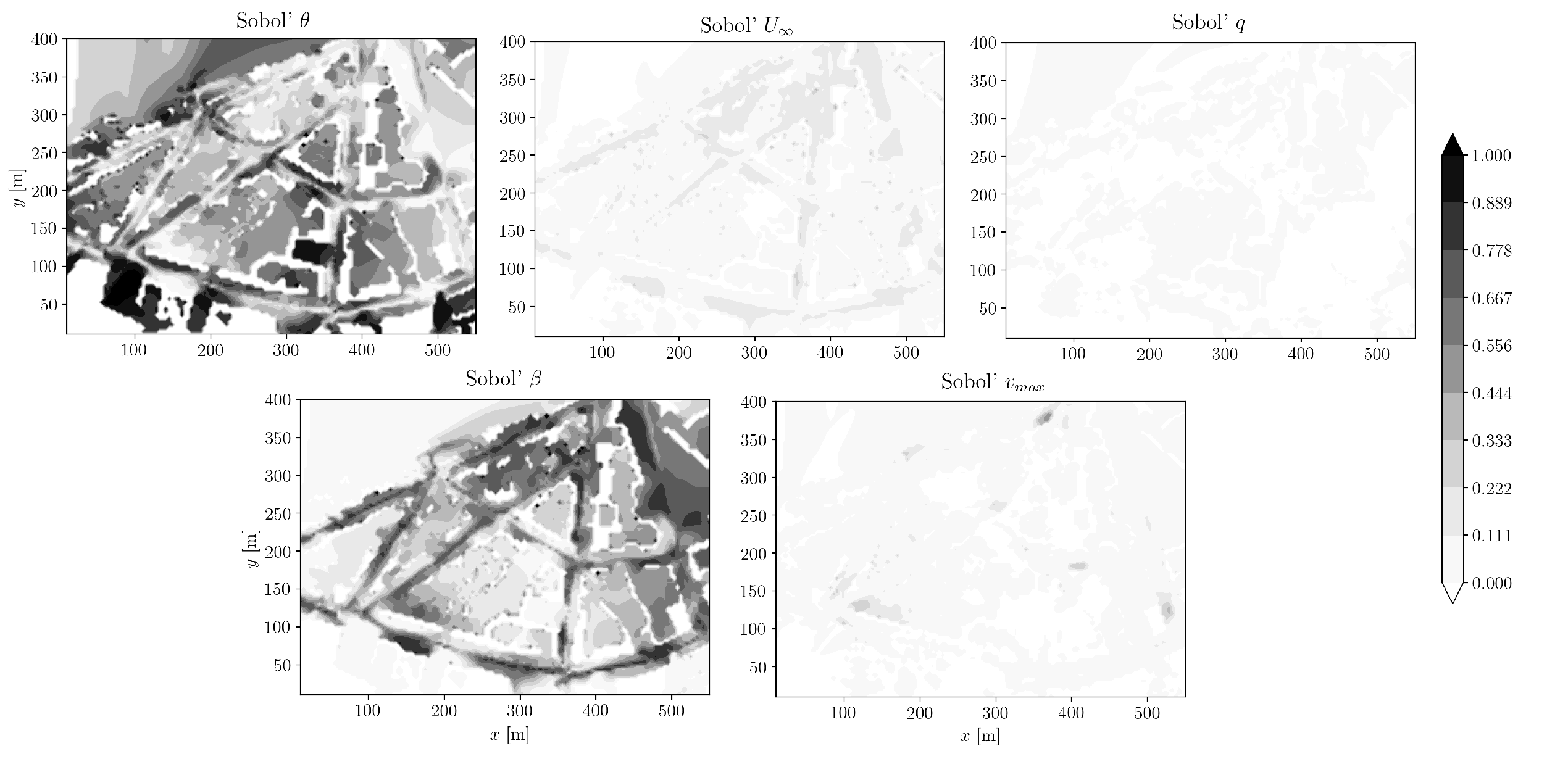}
    \caption{Two-dimensional maps of pointwise first-order Sobol' indices for time-averaged pollutant concentration, estimated by quasi-Monte-Carlo using $2^{10}$ sample points and Halton sequences.}
    \label{firstOrderSobolMaps}
\end{figure}

The information provided by such two-dimensional maps can be aggregated into scalar values by weighting the grid values of the first-order Sobol' indices with the corresponding estimated variances. They correspond to the vectorial Sobol' indices defined in \cite{vectorial_sobol}. Such \textit{generalised} Sobol' indices are defined as follows:
\[
    S^\text{gen}_{i} := \sum_{j=1}^m w_j S_{i}(x^{(j)}_1, x^{(j)}_2) \quad\text{with}\quad \quad w_j = \frac{\operatorname{Var}[\Phi_{\bm U}(x^{(j)}_1, x^{(j)}_2)]}{\sum_{k=1}^m \operatorname{Var}[\Phi_{\bm U}(x^{(k)}_1, x^{(k)}_2)]}.
\]

\begin{table}[h!]
    \centering
    \begin{tabular}{|c|c|c|c|c|}\hline
        $S^\text{gen}_\theta$ & $S^\text{gen}_{U_\infty}$ & $S^\text{gen}_q$ & $S^\text{gen}_\beta$ & $S^\text{gen}_{\nu_{max}}$ \\\hline
        $5.48\times10^{-1}$ & $4.47\times10^{-2}$ & $7.56\times10^{-4}$ & $3.22\times10^{-1}$ & $1.73\times10^{-2}$ \\\hline
    \end{tabular}
    \caption{Estimated generalised first-order Sobol' indices for each uncertain variable from two-dimensional sensitivity maps, using $2^{12}$ sample points from Halton sequences.}
    \label{estimGeneralizedSobolFirstOrder}
\end{table}

\begin{figure}
    \centering
    % Created by tikzDevice version 0.12.4 on 2023-10-31 11:01:04
% !TEX encoding = UTF-8 Unicode
\begin{tikzpicture}[x=1pt,y=1pt]
\definecolor{fillColor}{RGB}{255,255,255}
\path[use as bounding box,fill=fillColor,fill opacity=0.00] (0,0) rectangle (397.48,252.94);
\begin{scope}
\path[clip] (  0.00,  0.00) rectangle (397.48,252.94);
\definecolor{drawColor}{RGB}{255,255,255}
\definecolor{fillColor}{RGB}{255,255,255}

\path[draw=drawColor,line width= 0.6pt,line join=round,line cap=round,fill=fillColor] (  0.00,  0.00) rectangle (397.48,252.94);
\end{scope}
\begin{scope}
\path[clip] ( 46.55, 18.22) rectangle (391.98,247.44);
\definecolor{fillColor}{gray}{0.92}

\path[fill=fillColor] ( 46.55, 18.22) rectangle (391.98,247.44);
\definecolor{drawColor}{RGB}{255,255,255}

\path[draw=drawColor,line width= 0.3pt,line join=round] ( 46.55, 66.01) --
	(391.98, 66.01);

\path[draw=drawColor,line width= 0.3pt,line join=round] ( 46.55,139.88) --
	(391.98,139.88);

\path[draw=drawColor,line width= 0.3pt,line join=round] ( 46.55,213.75) --
	(391.98,213.75);

\path[draw=drawColor,line width= 0.6pt,line join=round] ( 46.55, 29.08) --
	(391.98, 29.08);

\path[draw=drawColor,line width= 0.6pt,line join=round] ( 46.55,102.95) --
	(391.98,102.95);

\path[draw=drawColor,line width= 0.6pt,line join=round] ( 46.55,176.82) --
	(391.98,176.82);

\path[draw=drawColor,line width= 0.6pt,line join=round] ( 86.41, 18.22) --
	( 86.41,247.44);

\path[draw=drawColor,line width= 0.6pt,line join=round] (152.84, 18.22) --
	(152.84,247.44);

\path[draw=drawColor,line width= 0.6pt,line join=round] (219.27, 18.22) --
	(219.27,247.44);

\path[draw=drawColor,line width= 0.6pt,line join=round] (285.70, 18.22) --
	(285.70,247.44);

\path[draw=drawColor,line width= 0.6pt,line join=round] (352.13, 18.22) --
	(352.13,247.44);
\definecolor{drawColor}{RGB}{210,105,30}
\definecolor{fillColor}{RGB}{210,105,30}

\path[draw=drawColor,line width= 0.4pt,line join=round,line cap=round,fill=fillColor] (219.27,231.30) circle (  2.50);

\path[draw=drawColor,line width= 0.4pt,line join=round,line cap=round,fill=fillColor] (352.13, 45.60) circle (  2.50);

\path[draw=drawColor,line width= 0.4pt,line join=round,line cap=round,fill=fillColor] (285.70, 29.36) circle (  2.50);

\path[draw=drawColor,line width= 0.4pt,line join=round,line cap=round,fill=fillColor] ( 86.41,148.03) circle (  2.50);

\path[draw=drawColor,line width= 0.4pt,line join=round,line cap=round,fill=fillColor] (152.84, 35.47) circle (  2.50);

\path[draw=drawColor,line width= 0.6pt,line join=round] (209.30,237.03) --
	(229.23,237.03);

\path[draw=drawColor,line width= 0.6pt,line join=round] (219.27,237.03) --
	(219.27,224.69);

\path[draw=drawColor,line width= 0.6pt,line join=round] (209.30,224.69) --
	(229.23,224.69);

\path[draw=drawColor,line width= 0.6pt,line join=round] (342.16, 53.09) --
	(362.09, 53.09);

\path[draw=drawColor,line width= 0.6pt,line join=round] (352.13, 53.09) --
	(352.13, 35.47);

\path[draw=drawColor,line width= 0.6pt,line join=round] (342.16, 35.47) --
	(362.09, 35.47);

\path[draw=drawColor,line width= 0.6pt,line join=round] (275.73, 29.96) --
	(295.66, 29.96);

\path[draw=drawColor,line width= 0.6pt,line join=round] (285.70, 29.96) --
	(285.70, 28.64);

\path[draw=drawColor,line width= 0.6pt,line join=round] (275.73, 28.64) --
	(295.66, 28.64);

\path[draw=drawColor,line width= 0.6pt,line join=round] ( 76.44,168.96) --
	( 96.37,168.96);

\path[draw=drawColor,line width= 0.6pt,line join=round] ( 86.41,168.96) --
	( 86.41,128.87);

\path[draw=drawColor,line width= 0.6pt,line join=round] ( 76.44,128.87) --
	( 96.37,128.87);

\path[draw=drawColor,line width= 0.6pt,line join=round] (142.87, 39.88) --
	(162.80, 39.88);

\path[draw=drawColor,line width= 0.6pt,line join=round] (152.84, 39.88) --
	(152.84, 31.50);

\path[draw=drawColor,line width= 0.6pt,line join=round] (142.87, 31.50) --
	(162.80, 31.50);
\end{scope}
\begin{scope}
\path[clip] (  0.00,  0.00) rectangle (397.48,252.94);
\definecolor{drawColor}{gray}{0.30}

\node[text=drawColor,anchor=base east,inner sep=0pt, outer sep=0pt, scale=  0.88] at ( 41.60, 26.05) {0.0};

\node[text=drawColor,anchor=base east,inner sep=0pt, outer sep=0pt, scale=  0.88] at ( 41.60, 99.92) {0.2};

\node[text=drawColor,anchor=base east,inner sep=0pt, outer sep=0pt, scale=  0.88] at ( 41.60,173.79) {0.4};
\end{scope}
\begin{scope}
\path[clip] (  0.00,  0.00) rectangle (397.48,252.94);
\definecolor{drawColor}{gray}{0.20}

\path[draw=drawColor,line width= 0.6pt,line join=round] ( 43.80, 29.08) --
	( 46.55, 29.08);

\path[draw=drawColor,line width= 0.6pt,line join=round] ( 43.80,102.95) --
	( 46.55,102.95);

\path[draw=drawColor,line width= 0.6pt,line join=round] ( 43.80,176.82) --
	( 46.55,176.82);
\end{scope}
\begin{scope}
\path[clip] (  0.00,  0.00) rectangle (397.48,252.94);
\definecolor{drawColor}{gray}{0.20}

\path[draw=drawColor,line width= 0.6pt,line join=round] ( 86.41, 15.47) --
	( 86.41, 18.22);

\path[draw=drawColor,line width= 0.6pt,line join=round] (152.84, 15.47) --
	(152.84, 18.22);

\path[draw=drawColor,line width= 0.6pt,line join=round] (219.27, 15.47) --
	(219.27, 18.22);

\path[draw=drawColor,line width= 0.6pt,line join=round] (285.70, 15.47) --
	(285.70, 18.22);

\path[draw=drawColor,line width= 0.6pt,line join=round] (352.13, 15.47) --
	(352.13, 18.22);
\end{scope}
\begin{scope}
\path[clip] (  0.00,  0.00) rectangle (397.48,252.94);
\definecolor{drawColor}{gray}{0.30}

\node[text=drawColor,anchor=base,inner sep=0pt, outer sep=0pt, scale=  0.88] at ( 86.41,  7.21) {$\beta$};

\node[text=drawColor,anchor=base,inner sep=0pt, outer sep=0pt, scale=  0.88] at (152.84,  7.21) {$\nu_{max}$};

\node[text=drawColor,anchor=base,inner sep=0pt, outer sep=0pt, scale=  0.88] at (219.27,  7.21) {$\theta$};

\node[text=drawColor,anchor=base,inner sep=0pt, outer sep=0pt, scale=  0.88] at (285.70,  7.21) {$q$};

\node[text=drawColor,anchor=base,inner sep=0pt, outer sep=0pt, scale=  0.88] at (352.13,  7.21) {$U_{\infty}$};
\end{scope}
\begin{scope}
\path[clip] (  0.00,  0.00) rectangle (397.48,252.94);
\definecolor{drawColor}{RGB}{0,0,0}

\node[text=drawColor,anchor=base,inner sep=0pt, outer sep=0pt, scale=  1.10] at ( 16.55,129.05) {$S^{gen}_i$};
\end{scope}
\end{tikzpicture}
    \caption{Estimated generalised first-order Sobol' indices for each uncertain variable from two-dimensional sensitivity maps, using $2^{12}$ sample points from Halton sequences and $100$ bootstrap resamples to construct $95\%$ confidence intervals.}
    \label{generalizedIndices}
\end{figure}
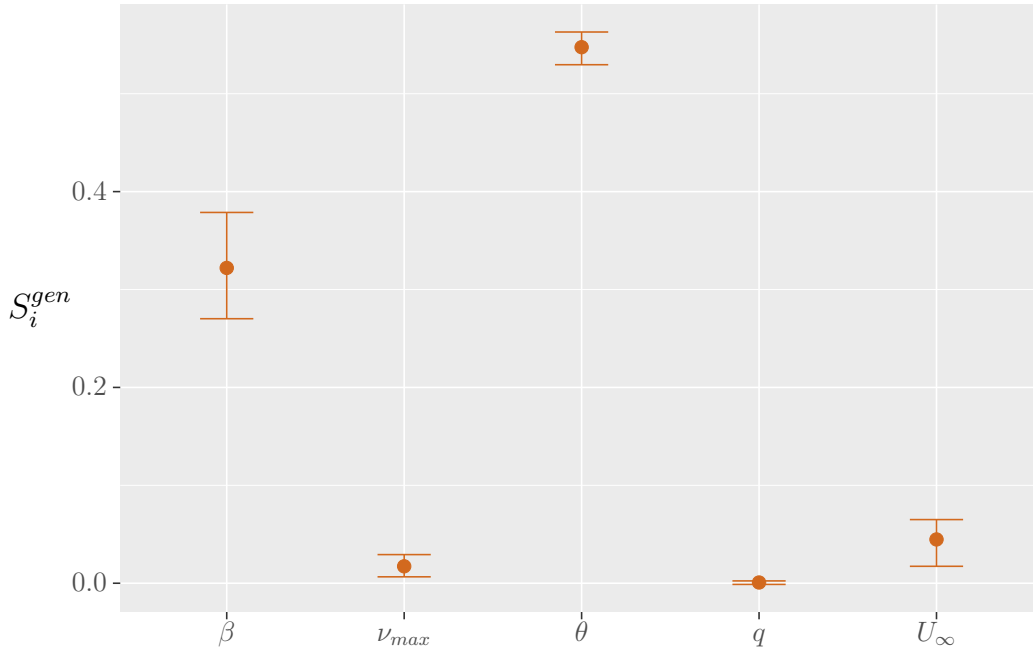

From the results shown in Figure \ref{firstOrderSobolMaps} and \ref{generalizedIndices} it can be seen that the wind angle and the rate of diesel engines in the fleet contribute to the major part of the variance of the output concentration field. The generalised indices reach approximately $55\%$ and $32 \%$ for these two variables respectively, while the other indices fall below $10\%$, and in particular the value of $S^\text{gen}_q$ is very close to zero, showing that the traffic volume $q$ is not influential in this study. The two-dimensional Sobol' maps highlight the fact that the sensitivity to the two dominant variables varies depending on the location in the urban neighbourhood: thus, if we choose to focus our analysis on a specific part of the domain - e.g. a specific building such as a hospital or a primary school - we may end up with a different hierarchy among the variables. We also carried out a second analysis (the results of which are not shown here) by fixing $\theta$ and $\beta$ to their mean values and evaluating the Sobol' maps for the three remaining variables. This showed that the wind intensity dominates over the speed limit (with a generalised index of about $70\%$ for the former and $30\%$ for the latter) and the traffic volume remains negligible. We can also observe that the indices sum to about $95\%$, which means that there are almost no interactions between the inputs.

% The fact that the wind angle is dominant could be expected from the fact that it determines the direction of the wind flow and thus the orientation of the plume in the domain. The significant contribution of $\beta$ is probably due to the nature of the emission model in the modelling chain, with an order of magnitude difference between diesel and petrol engine emission rates. As for the other variables, the wind magnitude is chosen to vary within a rather limited range, whereas in practice it could reach much smaller values. Our choice to exclude values below $5$ m/s is due to the fact that CFD simulations with slow winds require much more computational time to reach convergence, making the construction of a metamodel much more demanding. However, wind speed is expected to play a significant role, especially within the urban canopy where slow wind speeds lead to stagnation and higher concentrations of pollutants. The fact that the influence of $q$ is negligible is quite surprising, given its expected role in the observation of congestion and traffic jams. Our observations could be due to the choice of a too narrow probability distribution for this variable. Finally, the speed limit is observed to be influential in areas where the wind speed is rather low, whatever the values of the other parameters, and its influence on the output is difficult to explain.

Although this sensitivity analysis is performed point by point of the pollutant concentration map, it allows a first screening of the influencing variables and, in particular, generalised indices allow a quantitative ranking of the inputs. The next section presents another approach that considers the output as a subset of $\Real^3$ and performs a sensitivity analysis for such outputs.

\section{New sensitivity analysis indices for sets}

% Autres choix de titres
% \begin{itemize}
%     \item Set-based SA indices for pollutant concentration maps
%     \item Set-based indices for pollutant concentration maps Sensitivity Analysis
% \end{itemize}

\label{sec:set}
In this section we propose three types of indices designed to perform sensitivity analysis for sets, which we use on map-valued outputs. The first section \ref{subsec:notations2} introduces the notations to see map-valued models as examples of set-valued models. Then, three sections are dedicated to introduce the three indices. In section \ref{subsec:vorob}, we propose new indices based on Vorob'ev's theory of random sets (\cite{molchanov_random}). In section \ref{subsec:univ} we adapt the universal indices of \cite{GSA_Wass_universal_index} to sets. Finally, in Section \ref{subsec:kernel}, we present indices based on the Hilbert-Schmidt Independence Criterion (HSIC) in the context of set-valued output, as done in \cite{fellmann2023kernelbased}. 

\subsection{Notations}
\label{subsec:notations2}
 We want to perform a sensitivity analysis of a non-discretised map-valued model $\Phi$. To do this, we propose to consider $\Phi$ as a set-valued model $\Psi$ defined by:
$$
\begin{matrix}
 & \Udom & \longrightarrow & \mathcal L(\Xdom) \\ 
 \Psi : & \bm u & \mapsto & \Gamma_{\bm u}=\{(x_1, x_2,c) \in \Ddom \times \Cdom , c-\Phi_{\bm u}(x_1, x_2)\leq 0 \} 
\end{matrix}
$$
where $\Cdom=[c_{min}, c_{max}] $ and $\mathcal L(\Xdom)$ is the space of Lebesgue measurable subsets of $\Xdom=\Ddom \times \Cdom$. $\Gamma = \Psi(U)$ is then a random set that depends directly on the inputs $U_i$ whose effect is to be measured.  

For estimation purposes, $\Xdom$ is discretised with a grid of $m$ points denoted $(\bm x^{(1)},...,\bm x^{(m)})$. This allows the volume $\lambda(A)$ of $A \subset \Xdom$ to be estimated as $$\hat{\lambda}_m (A)= \frac{1}{m}\sum_{k=1}^{m} \mathbbold 1_{A}(\bm x^{(l)}).$$ Let also $(\bm U^{(i)},\Gamma^{(i)}), i=1,...,n$ be an independent and identically distributed (i.i.d.) sample of $(\bm U,\Gamma)$ with $\Gamma=\Psi({\bm U})$.

\subsection{Sensitivity analysis based on random set theory}
\label{subsec:vorob}
Working with sets instead of scalars or vectors makes the study of random elements much more difficult. For example, expectations and variance are not easy to define. In \cite{molchanov_random}, a complete theory of random sets is developed with definitions of expectations, medians and deviations of random sets. In particular, based on the Vorob'ev median and deviation, we propose to define indices inspired by Sobol' indices but adapted to random sets. 

When dealing with a scalar output $Y$, first-order Sobol' indices $S_i$ can be defined by :
$$S_i=1-\frac{ \Esp \left[ \operatorname{Var} [ Y | U_i ]\right] }{\operatorname{Var} Y}.$$
Here we want to define similar indices, but for set-valued outputs. To do this, we propose to replace the variance with the Vorob'ev median deviation, which is a possible adaptation of the median deviation to random sets. This requires defining the Vorob'ev median and the Vorob'ev median deviation, but also defining a Vorob'ev conditional median deviation.

\begin{definition}[Vorob'ev median]
    The Vorob'ev median $Q_{0.5}(\Gamma)$ of a random set $\Gamma$ is defined as the set of points whose probability of being in $\Gamma$ is at least $0.5$, i.e.
    $$
    Q_{0.5}(\Gamma)= \{x \in \Xdom, ~ \Proba ( x \in \Gamma) \geq 0.5 \}$$
\end{definition}

Using the Vorob'ev median, the Vorob'ev Median Deviation is defined as the Vorob'ev deviation between the random set and its median. We also propose conditional version of the Vorob'ev Median Deviation.
\begin{definition}
    The Vorob'ev Median Deviation of a random set $\Gamma$, denoted $\operatorname{VMD}(\Gamma)$, is defined by:
    $$
    \operatorname{VMD}(\Gamma)=\Esp [ \lambda ( \Gamma \Delta Q_{0.5}) ]. $$
    where $\Delta$ is the symmetric difference.
    The Vorob'ev conditional median deviation is the random variable $\operatorname{VMD}(\Gamma | U_i)$ defined by:
$$\operatorname{VMD}(\Gamma | U_i)=\Esp [ \lambda ( \Gamma \Delta Q^i_{0.5}) | U_i ]$$
where $Q^i_{0.5}$ is the Vorob'ev conditional median defined by:
$$
Q^i_{0.5}=\{ \bm x \in \Xdom, \Proba ( \bm x \in \Gamma | U_i ) \geq 0.5 \}. $$
\end{definition}

We can now define indices that describe the part of the Vorob'ev median deviation of a random set that is due to an input $U_i$.
\begin{definition}
Using the previous definition, we define a sensitivity index $S^i_V$ by:
  \begin{align*}
      S_i^{V}:&=1-\frac{\Esp \left[ \operatorname{VMD}(\Gamma | U_i) \right]}{\operatorname{VMD}(\Gamma)}\\
      &=1-\frac{\mathbb{E}[\lambda(\Gamma \Delta Q^i_{0.5})]}{\mathbb{E}[\lambda(\Gamma \Delta Q_{0.5})]}
  \end{align*}
 \end{definition}%
 This index quantifies the effect of $U_i$ on the output $\Gamma$, since it's zero if $\Gamma$ and $U_i$ are independent (simply as $Q^i_{0.5}=Q_{0.5}$ if $\Gamma$ and $U_i$ are independent). However, as first-order Sobol' indices, the converse doesn't hold. We show in Proposition \ref{prop_vorob} that it is between zero and one, as desired for sensitivity indices. The interpretation is close to a part of the total deviation explained by an input. However, there is no decomposition of the indices and they could sum to more than $1$.
 %which makes complex to theoretically justify either screening or ranking of the inputs using these indices. cette dernière partie de phrase est un peu trop négative à mon bout.  
\begin{prop}
\label{prop_vorob}
    $0 \leq S_i^{V} \leq 1$
\end{prop}
\begin{proof}
    It is clear that $S^i_{V}\leq 1$. To show that it's non-negative, we use the tower property of the conditional expectation on the numerator and the denominator :
    $$\frac{\mathbb{E}[\lambda(\Gamma \Delta Q^i_{0.5})]}{\mathbb{E}[\lambda(\Gamma \Delta Q_{0.5})]} = \frac{\mathbb E\mathbb{E}[\lambda(\Gamma \Delta Q^i_{0.5})|U_i]}{\mathbb E\mathbb{E}[\lambda(\Gamma \Delta Q_{0.5})|U_i]}.$$
    Then we show that $$\mathbb{E}[\lambda(\Gamma \Delta Q^i_{0.5})|U_i]\leq \mathbb{E}[\lambda(\Gamma \Delta Q_{0.5})|U_i]$$ almost surely, using the fact that the Vorob'ev median minimises the deviation for any measurable sets (see Proposition 2.4 of \cite{molchanov_random}). Finally by taking the expectation we conclude that $S_i^V$ is non-negative.
\end{proof}

 The estimation of $S_i^V$ is done by a double loop, since we are not aware of any cheaper estimator. Let $U_i^{(j)}, ~j \in (1,n)$ be $n$ independent copies of a given input $U_i$, the conditional median $Q^{i}_{0.5}(U_i^{(j)})=\{ x \in \Xdom, \Proba ( x \in \Gamma | U_i= U_i^{(j)}) \geq 0.5 \}$ is estimated by independently drawing another sample $U_{-i}^{(j,l))}, ~ l \in (1,n)$ of $U_{-i}$ and evaluating the output sets $\Gamma_{l}(U_i^{(j)})= \Psi(U_i^{(j)},U_{-i}^{(j,l)})$. The estimator of the conditional Vorob'ev median can then be derived:
$$
\hat{Q}^{i}_{0.5}(U_i^{(j)}):=\{ x \in \Xdom, \frac{1}{n} \sum_{l=1}^n \mathbbold 1 _{\Gamma_{l}(U_i^{(j)})}(x)\geq 0.5 \}
$$

Then $S_i^V$ is estimated by:
\begin{equation}
\label{eq:vorob}
\tilde{S}_i^V=1-\frac{\sum_{j=1}^{n} \hat{\lambda}_m ( \Gamma^{(j)} \Delta \hat{Q}^{i}_{0.5}(U_i^{(j)}))}{\sum_{j=1}^{n} \hat{\lambda}_m ( \Gamma^{(j)} \Delta \hat{Q}_{0.5})}.
\end{equation}
where $\Gamma^{(j)} = \Psi(U_i^{(j)}, U_{-i}^{(j)})$. 

The index is estimated by $n=32$ model evaluations per loop, resulting in a total of $n^2=1024$ model evaluations. Confidence intervals are obtained by bootstrapping with 100 resamples. The results are shown in Figure \ref{fig:vorob}.

\begin{figure}
    \centering
    % Created by tikzDevice version 0.12.4 on 2023-10-31 11:00:02
% !TEX encoding = UTF-8 Unicode
\begin{tikzpicture}[x=1pt,y=1pt]
\definecolor{fillColor}{RGB}{255,255,255}
\path[use as bounding box,fill=fillColor,fill opacity=0.00] (0,0) rectangle (397.48,252.94);
\begin{scope}
\path[clip] (  0.00,  0.00) rectangle (397.48,252.94);
\definecolor{drawColor}{RGB}{255,255,255}
\definecolor{fillColor}{RGB}{255,255,255}

\path[draw=drawColor,line width= 0.6pt,line join=round,line cap=round,fill=fillColor] (  0.00,  0.00) rectangle (397.48,252.94);
\end{scope}
\begin{scope}
\path[clip] ( 39.46, 18.22) rectangle (391.98,247.44);
\definecolor{fillColor}{gray}{0.92}

\path[fill=fillColor] ( 39.46, 18.22) rectangle (391.98,247.44);
\definecolor{drawColor}{RGB}{255,255,255}

\path[draw=drawColor,line width= 0.3pt,line join=round] ( 39.46, 60.03) --
	(391.98, 60.03);

\path[draw=drawColor,line width= 0.3pt,line join=round] ( 39.46,107.55) --
	(391.98,107.55);

\path[draw=drawColor,line width= 0.3pt,line join=round] ( 39.46,155.07) --
	(391.98,155.07);

\path[draw=drawColor,line width= 0.3pt,line join=round] ( 39.46,202.59) --
	(391.98,202.59);

\path[draw=drawColor,line width= 0.6pt,line join=round] ( 39.46, 36.27) --
	(391.98, 36.27);

\path[draw=drawColor,line width= 0.6pt,line join=round] ( 39.46, 83.79) --
	(391.98, 83.79);

\path[draw=drawColor,line width= 0.6pt,line join=round] ( 39.46,131.31) --
	(391.98,131.31);

\path[draw=drawColor,line width= 0.6pt,line join=round] ( 39.46,178.83) --
	(391.98,178.83);

\path[draw=drawColor,line width= 0.6pt,line join=round] ( 39.46,226.35) --
	(391.98,226.35);

\path[draw=drawColor,line width= 0.6pt,line join=round] ( 80.14, 18.22) --
	( 80.14,247.44);

\path[draw=drawColor,line width= 0.6pt,line join=round] (147.93, 18.22) --
	(147.93,247.44);

\path[draw=drawColor,line width= 0.6pt,line join=round] (215.72, 18.22) --
	(215.72,247.44);

\path[draw=drawColor,line width= 0.6pt,line join=round] (283.52, 18.22) --
	(283.52,247.44);

\path[draw=drawColor,line width= 0.6pt,line join=round] (351.31, 18.22) --
	(351.31,247.44);
\definecolor{drawColor}{RGB}{210,105,30}
\definecolor{fillColor}{RGB}{210,105,30}

\path[draw=drawColor,line width= 0.4pt,line join=round,line cap=round,fill=fillColor] ( 80.14,165.78) circle (  2.50);

\path[draw=drawColor,line width= 0.4pt,line join=round,line cap=round,fill=fillColor] (147.93, 41.12) circle (  2.50);

\path[draw=drawColor,line width= 0.4pt,line join=round,line cap=round,fill=fillColor] (215.72, 37.37) circle (  2.50);

\path[draw=drawColor,line width= 0.4pt,line join=round,line cap=round,fill=fillColor] (283.52,167.08) circle (  2.50);

\path[draw=drawColor,line width= 0.4pt,line join=round,line cap=round,fill=fillColor] (351.31, 44.26) circle (  2.50);

\path[draw=drawColor,line width= 0.6pt,line join=round] ( 69.97,237.03) --
	( 90.30,237.03);

\path[draw=drawColor,line width= 0.6pt,line join=round] ( 80.14,237.03) --
	( 80.14,115.16);

\path[draw=drawColor,line width= 0.6pt,line join=round] ( 69.97,115.16) --
	( 90.30,115.16);

\path[draw=drawColor,line width= 0.6pt,line join=round] (137.76, 83.79) --
	(158.10, 83.79);

\path[draw=drawColor,line width= 0.6pt,line join=round] (147.93, 83.79) --
	(147.93, 28.64);

\path[draw=drawColor,line width= 0.6pt,line join=round] (137.76, 28.64) --
	(158.10, 28.64);

\path[draw=drawColor,line width= 0.6pt,line join=round] (205.55, 68.07) --
	(225.89, 68.07);

\path[draw=drawColor,line width= 0.6pt,line join=round] (215.72, 68.07) --
	(215.72, 35.08);

\path[draw=drawColor,line width= 0.6pt,line join=round] (205.55, 35.08) --
	(225.89, 35.08);

\path[draw=drawColor,line width= 0.6pt,line join=round] (273.35,234.20) --
	(293.68,234.20);

\path[draw=drawColor,line width= 0.6pt,line join=round] (283.52,234.20) --
	(283.52,120.92);

\path[draw=drawColor,line width= 0.6pt,line join=round] (273.35,120.92) --
	(293.68,120.92);

\path[draw=drawColor,line width= 0.6pt,line join=round] (341.14, 80.69) --
	(361.48, 80.69);

\path[draw=drawColor,line width= 0.6pt,line join=round] (351.31, 80.69) --
	(351.31, 34.94);

\path[draw=drawColor,line width= 0.6pt,line join=round] (341.14, 34.94) --
	(361.48, 34.94);
\end{scope}
\begin{scope}
\path[clip] (  0.00,  0.00) rectangle (397.48,252.94);
\definecolor{drawColor}{gray}{0.30}

\node[text=drawColor,anchor=base east,inner sep=0pt, outer sep=0pt, scale=  0.88] at ( 34.51, 33.24) {0.0};

\node[text=drawColor,anchor=base east,inner sep=0pt, outer sep=0pt, scale=  0.88] at ( 34.51, 80.76) {0.1};

\node[text=drawColor,anchor=base east,inner sep=0pt, outer sep=0pt, scale=  0.88] at ( 34.51,128.28) {0.2};

\node[text=drawColor,anchor=base east,inner sep=0pt, outer sep=0pt, scale=  0.88] at ( 34.51,175.80) {0.3};

\node[text=drawColor,anchor=base east,inner sep=0pt, outer sep=0pt, scale=  0.88] at ( 34.51,223.32) {0.4};
\end{scope}
\begin{scope}
\path[clip] (  0.00,  0.00) rectangle (397.48,252.94);
\definecolor{drawColor}{gray}{0.20}

\path[draw=drawColor,line width= 0.6pt,line join=round] ( 36.71, 36.27) --
	( 39.46, 36.27);

\path[draw=drawColor,line width= 0.6pt,line join=round] ( 36.71, 83.79) --
	( 39.46, 83.79);

\path[draw=drawColor,line width= 0.6pt,line join=round] ( 36.71,131.31) --
	( 39.46,131.31);

\path[draw=drawColor,line width= 0.6pt,line join=round] ( 36.71,178.83) --
	( 39.46,178.83);

\path[draw=drawColor,line width= 0.6pt,line join=round] ( 36.71,226.35) --
	( 39.46,226.35);
\end{scope}
\begin{scope}
\path[clip] (  0.00,  0.00) rectangle (397.48,252.94);
\definecolor{drawColor}{gray}{0.20}

\path[draw=drawColor,line width= 0.6pt,line join=round] ( 80.14, 15.47) --
	( 80.14, 18.22);

\path[draw=drawColor,line width= 0.6pt,line join=round] (147.93, 15.47) --
	(147.93, 18.22);

\path[draw=drawColor,line width= 0.6pt,line join=round] (215.72, 15.47) --
	(215.72, 18.22);

\path[draw=drawColor,line width= 0.6pt,line join=round] (283.52, 15.47) --
	(283.52, 18.22);

\path[draw=drawColor,line width= 0.6pt,line join=round] (351.31, 15.47) --
	(351.31, 18.22);
\end{scope}
\begin{scope}
\path[clip] (  0.00,  0.00) rectangle (397.48,252.94);
\definecolor{drawColor}{gray}{0.30}

\node[text=drawColor,anchor=base,inner sep=0pt, outer sep=0pt, scale=  0.88] at ( 80.14,  7.21) {$\theta$};

\node[text=drawColor,anchor=base,inner sep=0pt, outer sep=0pt, scale=  0.88] at (147.93,  7.21) {$U_{\infty}$};

\node[text=drawColor,anchor=base,inner sep=0pt, outer sep=0pt, scale=  0.88] at (215.72,  7.21) {$q$};

\node[text=drawColor,anchor=base,inner sep=0pt, outer sep=0pt, scale=  0.88] at (283.52,  7.21) {$\beta$};

\node[text=drawColor,anchor=base,inner sep=0pt, outer sep=0pt, scale=  0.88] at (351.31,  7.21) {$\nu_{max}$};
\end{scope}
\begin{scope}
\path[clip] (  0.00,  0.00) rectangle (397.48,252.94);
\definecolor{drawColor}{RGB}{0,0,0}

\node[text=drawColor,anchor=base,inner sep=0pt, outer sep=0pt, scale=  1.10] at ( 13.01,129.05) {$S^{V}_i$};
\end{scope}
\end{tikzpicture}
    \caption{Estimation of $S^i_{V}$ for each input with 1024 model evaluations and a confidence interval estimated with $100$ bootstrap resamples.}
    \label{fig:vorob}
\end{figure}
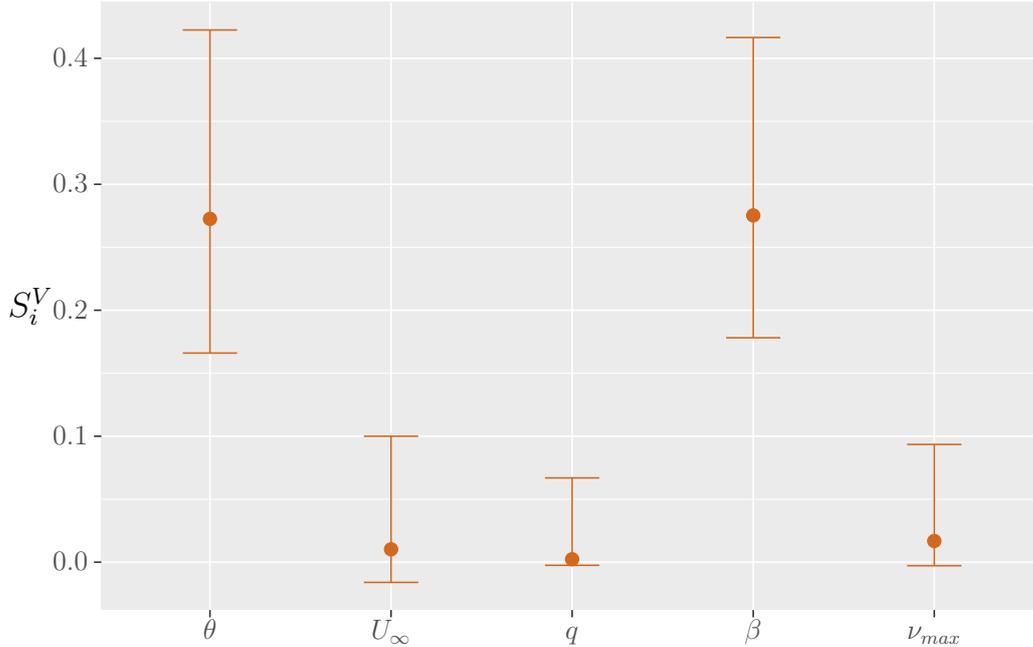
As these indices do not have a decomposition, it is difficult to justify a comparison between the value of each index. The only possible interpretation is that an index far from zero means that the input has an effect on the output. Here, the intervals of $\beta$ and $\theta$ do not include zero, which means that they have an effect. The other three, however, are close to $0$, implying that they may not be influential. The confidence intervals are also wide, which means that in another application we might fail to detect influential inputs. To become significant, additional model evaluations or a different estimation method is required. However, methods such as pick and freeze or rank-based estimation do not appear to be applicable in this scenario, as the quantity being estimated is not strictly a conditional variance.

\subsection{Universal sensitivity indices }
\label{subsec:univ}

In this second part, we propose to use \textit{universal} indices from \cite{GSA_Wass_universal_index}. In the latter paper, the authors define indices that can be applied to any metric output space. Using them with set-valued outputs, we define the following indices:
\begin{definition}
    $$S^{\text {Univ}}_{i}(\gamma,\mathbb Q)=\frac{\int_{0}^1 \operatorname{Var}\mathbb{E}\left[\lambda(\Gamma \Delta \gamma_{a}) \mid U_{i}\right] d \mathbb{Q}(a)}{\int_{0}^1 \operatorname{Var}\left(\lambda(\Gamma \Delta \gamma_a)\right) d \mathbb{Q}(a)}.$$ where $(\gamma_a)_{a \in [0,1]} \subset \Xdom$ is a collection of test sets parameterised by $a \in [0,1]$
    and $\mathbb Q$ is a probability distribution on $[0, 1]$ to choose.
\end{definition}%
The idea behind these indices is to quantify the effect of inputs on multiple scalar transformations of the set-valued output, rather than on the set-valued output itself. Specifically, we measure this variability by the volume of the symmetric difference between the output set and a collection of test sets.

Similar to Sobol' indices, the contribution of each input and their interactions can be decomposed to rank the inputs. Indeed, the decomposition can be obtained by taking the Sobol-Hoeffding decomposition of $\operatorname{Var}\left(\lambda(\gamma \Delta \gamma_a)\right)$ for each $a$ and then summing for $0$ to $1$ with respect to the measure $\mathbb Q$. An independence test is not accessible, mainly because the family of all transformations does not necessarily characterise the whole distribution of the set-valued output. However, screening can still be done, either by keeping only the inputs above a certain threshold, or by keeping the first inputs in the ranking.

After selecting the test sets and the law $\mathbb Q$, these indices can be estimated in a similar way to Sobol's indices. We use a rank-based estimator, which is recalled in \cite{GSA_Wass_universal_index}. Plugging in the volume estimation, the estimator $\widehat{S}^{\text {Univ}}_{i}$ of $S^{\text {Univ}}_{i}$ is given by the ratio between
\begin{align}
\label{eq:univnum}
    \widehat{S}_{num}=&\frac{1}{N_a}\sum_{l=1}^{N_a} \left[\frac{1}{n}\sum_{j=1}^n \left( \hat{\lambda}_m(\Gamma^{(j)} \Delta \gamma_{a_l})\right) \left( \hat{\lambda}_m(\Gamma^{(N_i(j))} \Delta \gamma_{a_l})\right) \right]\\
    &-\frac{1}{N_a}\sum_{l=1}^{N_a} \left[\frac{1}{n}\sum_{j=1}^n \left( \hat{\lambda}_m(\Gamma^{(j)} \Delta \gamma_{a_l})\right) \right]^2 \nonumber
\end{align}
and,
\begin{align}
\label{eq:univden}
   \widehat{S}_{den}= \frac{1}{N_a}\sum_{l=1}^{N_a} \left[\frac{1}{n}\sum_{j=1}^n \left( \hat{\lambda}_m(\Gamma^{(j)} \Delta \gamma_{a_l})\right)^2 \right]-\frac{1}{N_a}\sum_{l=1}^{N_a} \left[\frac{1}{n}\sum_{j=1}^n \left( \hat{\lambda}_m(\Gamma^{(j)} \Delta \gamma_{a_l})\right) \right]^2,
\end{align}
where $(a_l)$ is an iid sample of the law $\mathbb Q$ and $N_i(j)$ is the index in the sample $(U_i^l)$ that comes after $U_i^j$ when $(U_i^l)$ is sorted in ascending order (see \cite{rank_sobol} for details).

We compute the indices with $n=1000$ model evaluations and $N_a=100$ test sets. To remain general, $\Xdom$ is rescaled to $[0,1]^3$ which is centered on $x^0=(\frac{1}{2},\frac{1}{2},\frac{1}{2})$. We test four families of test sets defined on $[0,1]^3$:
\begin{itemize}
    \item Centered balls: $\gamma_a=B(x^0,a)$ with $\mathbb Q \sim \Udom ([0,\frac{1}{2}])$
    \item Centered squares:  $\gamma_a=\{x \in [0,1]^{3} , ||x-x^0||_{\infty}\leq a\}$ with $\mathbb Q \sim \Udom ([0,\frac{1}{2}])$
    \item Slides along the i-th dimension: $\gamma_a^i = \{x \in [0,1]^{3}, x^i \leq a \}$ with $\mathbb Q \sim \Udom ([0,1])$
    \item Vorob'ev quantiles: $\gamma_a= \{x\in [0,1]^{3}, \Proba(x \in \Gamma) \geq a \}$ with $\mathbb Q \sim \mathcal N(\frac{1}{2},\sigma^2)$
\end{itemize}
The slides are used with $i=3$ to get levels of concentration test sets. Confidence intervals are obtained by bootstrapping with $B=100$ resamples. We use bootstrap samples of size $0.8n=800$ without replacement because replacement introduces high bias in rank-based estimation. The intervals are adjusted by a correction factor. The results are displayed in Figure \ref{fig:univ}.

\begin{figure}
    \centering
    \input{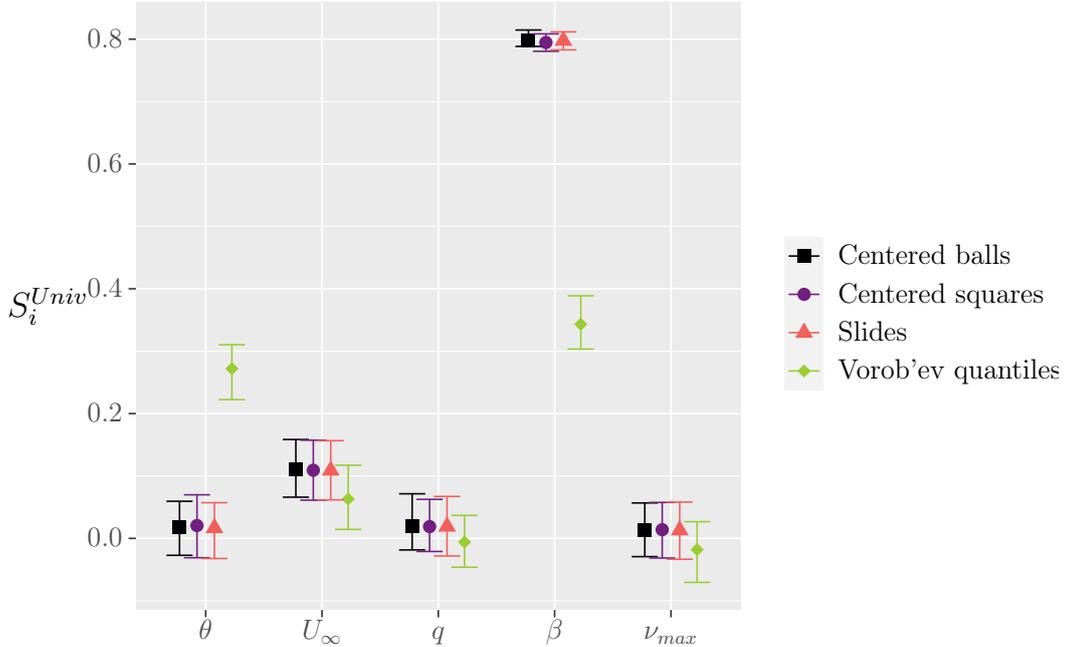}
    \caption{Estimation of the universal indices ${S}^{\text {Univ}}_{i}$ for each input and for four different test sets and $N_a=100$, with 1000 model evaluations. Confidence intervals are obtained with $100$ bootstrap samples}
    \label{fig:univ}
\end{figure}

The results are comparable when using centered balls, centered squares, and slices. $\beta$ is the most dominant input with an index value of about $80\%$. The value of $U_{\infty}$ is approximately $10\%$, and all other inputs have minimal impact. In contrast, when Vorob'ev quantiles are the test sets, $\beta$ and $\theta$ are the most important inputs, and $U_{\infty}$ also has some influence to a lesser extent. The output is clearly not influenced by either $q$ or $\nu_{max}$. This difference shows that the choice of test sets is crucial, as it has a strong influence on the values of the indices. When the law $\mathbbold Q$ of Vorob'ev quantiles is approximately $0.5$, their geometry is similar to $\Gamma$'s realisations. Therefore, changes in $\Gamma$ can be accurately detected by measuring the volume of the symmetric difference. Conversely, it may not be fully possible to detect all variations of the random set by comparing realisations of $\Gamma$ with simple sets as squares, cubes or slices.
%In the next section the universal indices are computed using Vorob'ev quantiles as test sets.  

\subsection{HSIC-based sensitivity analysis}
\label{subsec:kernel}
In this section, we propose to define kernel-based sensitivity indices with set-valued outputs, as the use of kernels makes such indices very permissive in the type of considered outputs (or inputs).

Sensitivity analysis, based on measures of dependence, consists in examining the influence of an input $X_i$ on an output $Y$ by measuring their mutual dependence. To do this, a distance is calculated between the joint distribution $\Proba_{X_i,Y}$ and the product of their marginal distributions $\Proba_{X_i} \otimes \Proba_Y$. If this distance is zero, then $X_i$ and $Y$ are independent, meaning $X_i$ has no impact on $Y$. A frequently used measure for dependence is the Hilbert-Schmidt Independence Criterion (HSIC). It operates by selecting appropriate input and output kernels, and subsequently measuring the squared distance between the mean embeddings of $\Proba_{X_i,Y}$ and $\Proba_{X_i} \otimes \Proba_Y$ in the corresponding Reproducing Kernel Hilbert Space (RKHS), see \cite{Gretton_hsic1} for a proper definition of the tools needed to define the HSIC. This calculation can be performed using just one sample, which is among the reasons why HSICs have gained popularity.

If a \textit{characteristic} kernel is available on any space, meaning a kernel whose mean embedding in the RKHS is injective, then HSIC-based indices can be used. Consequently, their application is appropriate for sensitivity analysis of set-valued models. \cite{fellmann2023kernelbased} introduced a kernel $k_{set}$ defined on sets, which is proved to be characteristic. It is then used to define HSIC-based indices on sets. The latter definition requires to have
\begin{itemize}
    \item  a characteristic kernel on $\mathcal L (\Xdom)$. The proposed kernel in \cite{fellmann2023kernelbased} is
    $$k_{set}(\gamma_1,\gamma_2)=\exp\left(- \frac{\lambda(\gamma_1 \Delta \gamma_2)}{2\sigma^2} \right) ~~ \forall \gamma_1,\gamma_2 \subset \Xdom, $$
    where $\sigma^2$ is a normalization constant.
    \item  $p$ characteristic kernels $K_i$ on the input space $\Udom_i$. As shown in \cite{daveiga:hal-03108628}, an ANOVA-like decomposition of HSIC exists if the input kernels are ANOVA, i.e. they can be written as $1+k_i$, where $k_i$ is a kernel whose induced RKHS consists of zero-mean functions. Sobolev kernels are the most commonly used kernels that fulfil these requirements. Furthermore, classic characteristic kernels can be modified to become ANOVA, as described in \cite{anova_kernel}.

\end{itemize}
\begin{definition}
    With the previous kernels, HSIC-based indices on sets are defined by 
    $$\forall i \in \{1,...,p\},~ S^{\operatorname{H}_{set}}_i :=\frac{\operatorname{H}_{set}(U_i,\Gamma)}{\operatorname{H}_{set}(\bm U,\Gamma)},$$
    where 
$$\operatorname{H}_{set}(U_i,\Gamma)= \Esp \left[ (K_i(U_i,{U_i}')-1)k_{set}(\Gamma,\Gamma') \right],$$
$$\operatorname{H}_{set}(\bm U,\Gamma)= \Esp \left[ (K(\bm U,{\bm U}')-1)k_{set}(\Gamma,\Gamma') \right],$$ 
with 
$$K=\bigotimes_{j=1}^p K_j,$$
and $(\bm U', \Gamma')$ an independent copy of $(\bm U, \Gamma)$.
\end{definition}

Since these indices are zero if and only if the considered input $U_i$ and the output $\Gamma$ are independent, they are suitable for screening. Using this result, p-values associated with the independent tests can be calculated to screen the inputs. Besides, we choose the input kernels $K_i$ as ANOVA kernels so that the indices $S_i^{H_{set}}$ have an ANOVA decomposition as done in \cite{fellmann2023kernelbased}. Subsequently, the inputs can be ranked by their impact based on each index's value.

The indices can be estimated using:
\begin{equation}
\label{eq:hsicestim}
\widehat{\widehat{\operatorname{H}_{set}}}\left(U_i, \Gamma\right)=\frac{2}{n(n-1)} \sum_{j < l}^{n}\left(K_{i}\left(U_i^{(j)}, U_i^{(l)}\right)-1\right) \exp(-\frac{\lambda(\Xdom)}{2\sigma^2}\hat{\lambda}_m(\Gamma^{(j)} \Delta \Gamma^{(l)})).
\end{equation}

We compute the indices for each of the five inputs using a sample size of $n=1000$. The results for five different input ANOVA kernels are compared using the following kernels:
\begin{itemize}
    \item the Sobolev kernel of order $1$, $k_{sob}(x,y)= 1+(x-\frac{1}{2})(y-\frac{1}{2})+\frac{1}{2}[(x-y)^2-|x-y| +\frac{1}{6}]$
    \item the Gaussian kernel, $k_{rbf} (x,y) = e^{-\frac{1}{2} \left( \frac{x-y}{\sigma}\right)^2}$ with $\sigma > 0$,
    \item the Laplace kernel, $k_{exp} (x,y) =e^{-\frac{|x-y|}{h}}$ with $h>0$,
    \item the Matérn 3/2, $k_{3/2} (x,y) =\left( 1+\sqrt 3 \frac{|x-y|}{h} \right) e^{- \sqrt 3 \frac{|x-y|}{h}}$ with $h >0$,
    \item the Matérn 5/2, $k_{5/2} (x,y) =\left( 1+\sqrt 5\frac{|x-y|}{h}+ \frac{5}{3} \frac{|x-y|}{h^2} \right) e^{- \sqrt 5 \frac{|x-y|}{h}}$ with $h >0$.
\end{itemize}
The Sobolev kernel is already ANOVA, whereas four additional kernels need adaptations to become ANOVA (see \cite{anova_kernel} for the detailed transformation). Confidence intervals are estimated using bootstrap methods with $B=100$ resamples. The value of $\sigma$ is taken as the empirical median of $\hat{\lambda}_m(\Gamma^{(j)} \Delta \Gamma^{(l)})$ for $j<l$. Figure \ref{hsic_comp} shows our results. 
\begin{figure}[]
    \centering
    \input{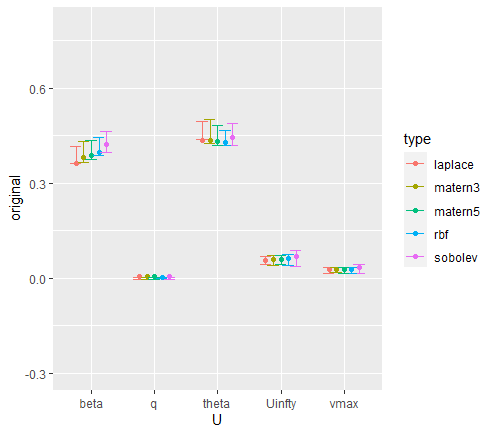}
    \caption{Estimation of $S^{\operatorname{H}_{set}}_i$ for five input kernels, 1000 model evaluations. Confidence intervals are obtained by bootstrap with $100$ resamples}
    \label{hsic_comp}
\end{figure}
We observe that the choice of the input kernel does not significantly affect the values of the indices: regardless of the kernel chosen, $\beta$ and $\theta$ remain the most influential variables (with about $40\%$ each) and the other three inputs the least influential (with about $10\%$). The input ranking can be used for screening purposes. For instance, one may decide to keep only the top three most influential inputs. However, using HSIC-based indices guarantees trustworthy screening by computing the p-values of independence tests between an input and an output. This is done in Table \ref{tab:pval}, where the p-values for each input are computed by asymptotic estimation (\cite{song_supervised_2007}) for each input kernel. 
\begin{table}[]
\centering
\begin{tabular}{|c|c|c|c|c|c|c|}
\hline
& $\beta$ & $q$ & $\theta$ & $U_{\infty}$ & $\nu_{max}$ \\ \hline
$k_{sob}$ & $< 10^{-16}$ & $0.74 \pm 0.6$ & $< 10^{-16}$ & $< 10^{-16}$ & $< 10^{-16}$ \\ \hline
$k_{rbf}$ & $< 10^{-16}$ & $0.76 \pm 0.6$ & $< 10^{-16}$ & $< 10^{-16}$ & $< 10^{-16}$ \\ \hline
$k_{exp}$ & $< 10^{-16}$ & $0.89 \pm 0.5$ & $< 10^{-16}$ & $< 10^{-16}$ & $< 10^{-16}$ \\ \hline
$k_{3/2}$ & $< 10^{-16}$ & $0.82 \pm 0.6$ & $< 10^{-16}$ & $< 10^{-16}$ & $< 10^{-16}$ \\ \hline
$k_{5/2}$ & $< 10^{-16}$ & $0.80 \pm 0.6$ & $< 10^{-16}$ & $< 10^{-16}$ & $< 10^{-16}$ \\ \hline
\end{tabular}
\caption{Asymptotic estimation of the p-values of HSIC-ANOVA-based independence test with five input kernel and $1000$ model evaluations and $100$ bootstrap samples.}
\label{tab:pval}
\end{table} If the p-value exceeds 0.05, then the input has no influence on the output. Specifically, in this case, we observe that only $q$, for each kernel, has a p-value above 0.05. Additionally, a standard deviation is estimated for $q$ through bootstrap analysis, which is given in the second column. We confirm that despite this deviation, the p-values of $q$ remain above 0.05. For the remaining four inputs, p-values below $10^{-16}$ were obtained, and standard deviations of the same scale indicate that these four inputs are influential. The chosen input kernel slightly modifies the p-values in this case, but it does not affect the screening results. However, in other test cases, the choice of kernel could impact the results, which is the primary limitation of the HSIC-based indices. Some research has also been conducted to aid in kernel selection. For instance, in \cite{amri_morel}, the authors suggest selecting the kernel that maximises test power. Without further investigation into the impact of kernels, the Sobolev kernel is preferable as it does not require any choice of hyperparameters.

% In terms of choosing between the five input kernels, we can say that the Sobolev kernel (à définir) is an attractive choice due to its lack of parameter requirements. Another approach to choosing a kernel can be to select the one that maximises the test power, as recommended in \cite{amri_morel} {\color{red} Je ne suis pas convaincue par ce paragraphe technique alors que tu n'as pas introduit les noyaux ni présenté les paramètres associés. Ici c'est lequel qui maximise le test power ? }.

%%----Cpmparison----%%

\section{Comparison of the indices }
 \label{sec:comparison}
In this section we compare the four indices recalled here :
\begin{itemize}
    \item the generalised first-order Sobol' indices $S^{gen}_i$;
    \item the Vorob'ev based indices $S^V_i$;
    \item the universal indices $S^{\text {Univ}}_{i}$ computed using Vorob'ev quantiles as test sets;
    \item the HSIC-based indices $S^{\operatorname{H}_{set}}_i$ using the Sobolev kernel as ANOVA input kernel.
\end{itemize}
Each index is computed with a budget of about $1000$ model evaluations. $S^{gen}_i$ are estimated with the pick and freeze estimator given in (\ref{eq:pickandfreeze}). $S_i^V$ are estimated using a double loop method as described in (\ref{eq:vorob}), with $32$ model evaluations in each loop, resulting in $1024$ evaluations. The universal indices are evaluated using the rank-based estimation approach described in (\ref{eq:univnum}) and (\ref{eq:univden}), with $N_a=100$ test sets. Lastly, the estimate for $S^{\operatorname{H}_{set}}_i$ is given by (\ref{eq:hsicestim}). 

The results are shown in Figure \ref{comp1000}.
\begin{figure}
    \centering
    \input{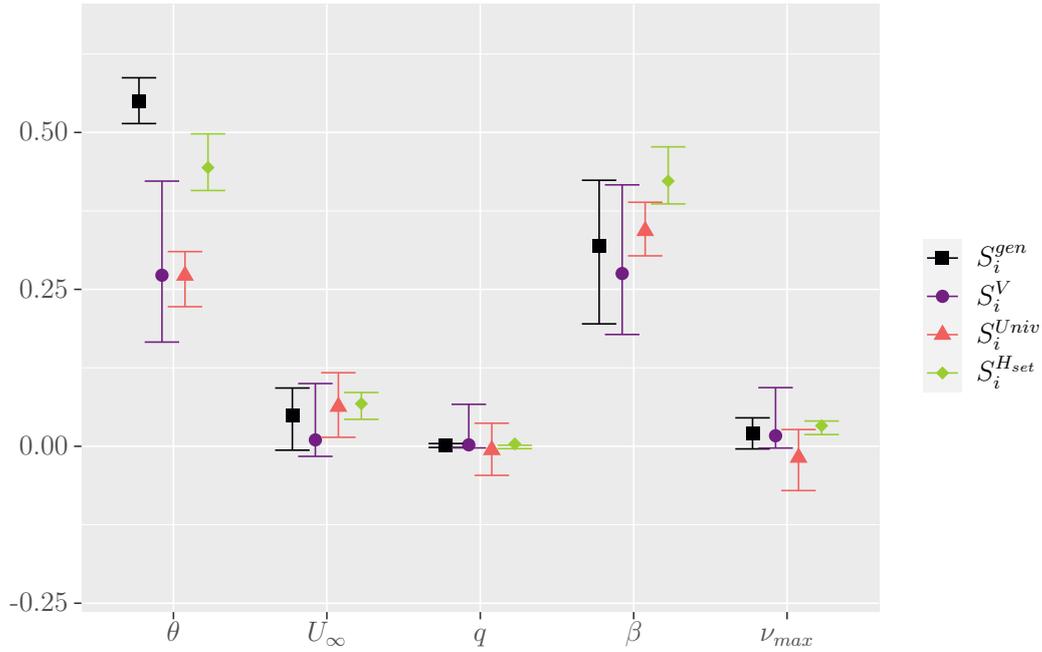}
    \caption{Comparison of the four indices with a total budget of $n= 1000$ model evaluations. $100$ bootstrap sample are used to estimate confidence intervals}
    \label{comp1000}
\end{figure}
First we check that the four indices give similar results. The most influential inputs are $\theta$ and $\beta$, and the least influential are $U_{\infty}$, $q$ and $\nu_{max}$. However, there are differences in the obtained results between the indices regarding screening and ranking.
The HSIC is the only index that allows access to a reliable independence test, which categorizes $q$ as non-influential in screening (according to the p-values given in Table \ref{tab:pval}). For the other three indices, a possible screening method is to select a threshold under which inputs will be deemed non-influential. For instance, if a 10\% threshold is chosen, then $U_{\infty}$, $q$, and $\nu_{max}$ would be considered negligible. However, with a smaller threshold, conclusions would be challenging to draw as some confidence intervals overlap. 
The four significant inputs, identified by HSIC-based tests, can be confidently ranked since their confidence intervals do not overlap. In contrast, the ranking of the remaining inputs varies for each index. Similar to $S_i^{H_{set}}$, $S_i^V$ places $\theta$ and $\beta$ at the same level. However, $\theta$ is much more influential than $\beta$ according to the generalised Sobol' indices, and the opposite is true for the universal indices. The ranking is justified in theory for all indices except for $S_i^V$, which lacks any decomposition.

Examining the confidence intervals, all four indices remain satisfactory with $1000$ model evaluations, but the HSIC-based indices, and to a lesser extent the universal indices, have much less variability. To emphasise this disparity, we create an identical plot on Figure \ref{comp100} with only $100$ model evaluations ($10$ in each loop for $S_i^V$). 
\begin{figure}
    \centering
    \input{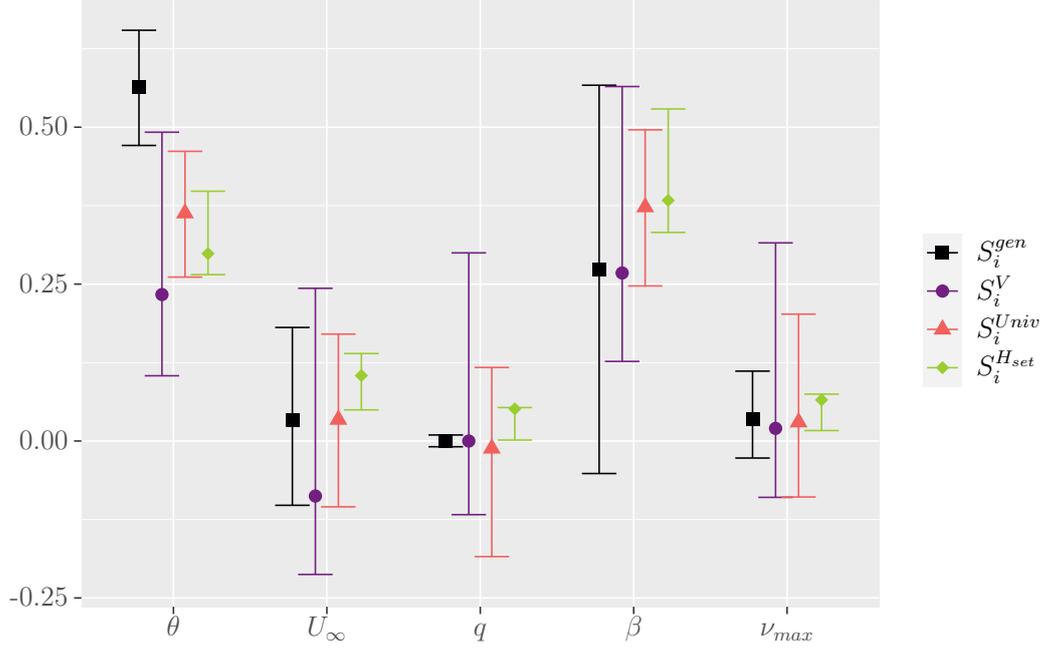}
    \caption{Comparison of the four indices with a total budget of $n= 100$ model evaluations. $100$ bootstrap sample are used to estimate confidence intervals}
    \label{comp100}
\end{figure}
This time the HSIC-based indices can be used as they have relatively small confidence intervals. The same conclusion in terms of screening and ranking can be reached with a single $100$-sample. The universal indices can also still be used in this case, but their variability makes it possible to have non-positive indices. As for the other two indices, it is important to note that they have not yet converged, rendering them unusable with only $100$ model evaluations.

The strengths and limitations of each index are summarised in Table \ref{tab:comp}. 

 \begin{table}
\centering
\begin{tabular}{|l|m{2.5cm}|m{2.5cm}|m{2.5cm}|m{2.5cm}|}
\cline{2-5}  
\multicolumn{1}{c|}{} 
& \textbf{Ranking} & \textbf{Screening} & \textbf{Evaluations} & \textbf{Limitations} \\
\hline
$S^{gen}_i$ & \ding{51} ANOVA & $\bm \sim$ threshold & $\bm \sim (p+1)n$ & \ding{55} pointwise influence \\
\hline
$S_i^V$ & \ding{55} no decomposition & \ding{55} no screening method & \ding{55} $n^2 $ (double loop)& \ding{51} no choice to be made \\
\hline
$S^{\text {Univ}}_{i}$ & \ding{51} ANOVA & $\bm \sim$ threshold & $\bm \sim$ $n$ but big confidence intervals & \ding{55} choice of the test sets  \\
\hline
$S^{\operatorname{H}_{set}}_i$ & \ding{51} ANOVA & \ding{51} independence test & \ding{51}  $n$ and small confidence intervals & $\bm \sim$ choice of kernel\newline
\ding{55} Interactions interpretations\\
\hline
\end{tabular}
\caption{Indices comparison}
\label{tab:comp}
\end{table}

\section{Conclusion}
\label{sec:conclusion}
In this paper, we aim to quantify the impact of inputs on a map-valued model, specifically a concentration map model for a realistic urban setting in the greater Paris region (France). To achieve this goal, we propose a sensitivity analysis of the model. As a starting point, we perform a pointwise sensitivity analysis on each point of the map. This produces maps of Sobol' indices and allows us to observe the influence of each input in each zone of the map. Generalised indices are also derived to synthesise the overall impact of each input. A second approach deals with the model as a whole. Three types of indices, defined for set-valued outputs, are introduced and used with the map-valued outputs. The first is based on quantities from random set theory. The second is based on universal indices that can be used with set outputs. Finally, the last relies on kernel-based sensitivity analysis and, in particular, on HSIC-based indices. A comparison of each index reveals that the HSIC-based index is the most efficient in terms of the number of model evaluations and can be used for both screening and ranking. However, it may not be the most suitable option if we are interested in examining the interactions between inputs. In such situations, universal indices and generalized Sobol' may be more appropriate.

The introduction of sensitivity indices for sets was initially devised for their application to maps. However, their utility extends beyond maps and can be effectively employed for any model with set-valued outputs, like for example sets of feasible points in mechanical engineering. While we specifically addressed the two-dimensional case to directly handle maps, it is also worth noting that theses indices can be used with $\Xdom$ of any finite dimension.

\section{Acknowledgment}
This research was conducted with the support of the consortium in Applied Mathematics
\href{https://doi.org/10.5281/zenodo.6581216   }{\underline{CIROQUO}}, gathering partners in technological research and academia in the development of
advanced methods for Computer Experiments.

\setlength\bibitemsep{0.5\itemsep}
\printbibliography
\end{document}